\documentclass[12pt,reqno]{amsart}
\usepackage{amssymb,citesort}
\renewcommand*{\leq}{\leqslant} \renewcommand*{\geq}{\geqslant}

\newcommand*{\vare}{\varepsilon}
\newcommand*{\sqvare}{\sqrt{\vare}} \newcommand*{\sqsq}{\sqvare\,}
\newcommand*{\vart}{\vartheta} \newcommand*{\varp}{\varphi}

\newcommand*{\mC}{\mathbb C} \newcommand*{\mR}{\mathbb R}
\newcommand*{\mT}{\mathbb T} \newcommand*{\mZ}{\mathbb Z}

\newcommand*{\cA}{\mathcal A} \newcommand*{\cB}{\mathcal B}
\newcommand*{\bcC}{\text{\boldmath $\mathcal C$}}
\newcommand*{\cF}{\mathcal F} \newcommand*{\cK}{\mathcal K}
\newcommand*{\cL}{\mathcal L} \newcommand*{\cM}{\mathcal M}
\newcommand*{\cN}{\mathcal N} \newcommand*{\cO}{\mathcal O}
\newcommand*{\cR}{\mathcal R} \newcommand*{\cS}{\mathcal S}
\newcommand*{\cT}{\mathcal T} \newcommand*{\cY}{\mathcal Y}
\newcommand*{\cZ}{\mathcal Z}

\newcommand*{\fG}{\mathfrak G} \newcommand*{\ZfG}{\widehat{\fG}}
\newcommand*{\fH}{\mathfrak H}
\newcommand*{\fL}{\mathfrak L} \newcommand*{\ZfL}{\widehat{\fL}}
\newcommand*{\fM}{\mathfrak M} \newcommand*{\fN}{\mathfrak N}
\newcommand*{\fR}{\mathfrak R} \newcommand*{\fT}{\mathfrak T}
\newcommand*{\fZ}{\mathfrak Z}

\newcommand*{\sla}{\text{\textsl{a}}} \newcommand*{\slb}{\text{\textsl{b}}}
\newcommand*{\slu}{\text{\textsl{u}}} \newcommand*{\slv}{\text{\textsl{v}}}
\newcommand*{\slw}{\text{\textsl{w}}}

\newcommand*{\Tf}{\widetilde{f}} \newcommand*{\Tg}{\widetilde{g}}
\newcommand*{\Th}{\widetilde{h}}

\newcommand*{\bc}{\mathbf c} \newcommand*{\br}{\mathbf r}
\newcommand*{\bC}{\mathbf C} \newcommand*{\bL}{\mathbf L}
\newcommand*{\bM}{\mathbf M} \newcommand*{\bS}{\mathbf S}
\newcommand*{\bOmega}{\mathbf{\Omega}}

\newcommand*{\gl}{\mathsf{gl}} \newcommand*{\GL}{\mathsf{GL}}

\newcommand*{\const}{\mathrm{const}}

\DeclareMathOperator{\ad}{ad} \DeclareMathOperator{\codim}{codim}
\DeclareMathOperator{\diag}{diag} \DeclareMathOperator{\rank}{rank}
\DeclareMathOperator{\Ad}{Ad} \DeclareMathOperator{\Fix}{Fix}
\DeclareMathOperator{\Image}{Image} \DeclareMathOperator{\Ker}{Ker}

\newcommand*{\new}{^{\text{new}}}

\newtheorem{lem}{Lemma} \newtheorem{thm}{Theorem}

\theoremstyle{definition}
\newtheorem{defn}{Definition}
\newtheorem*{cond}{Condition~$\Omega$} \newtheorem*{nott}{Notation~$\fT$}

\allowdisplaybreaks

\begin{document}

\title[Lower dimensional tori within the reversible context~2]%
{KAM theory for lower dimensional tori \\ within the reversible context~2}

\author[M.~B.~Sevryuk]{Mikhail B. Sevryuk}

\address{Institute of Energy Problems of Chemical Physics, The Russia Academy
of Sciences, Leninski\u{\i} prospect~38, Bldg.~2, Moscow 119334, Russia}

\email{sevryuk@mccme.ru}

\dedicatory{To the memory of Vladimir Igorevich Arnold who is so unexpectedly
gone}

\thanks{This study was partially supported by a grant of the President of the
Russia Federation, project No.\ NSh-8462.2010.1.}



\begin{abstract}
The reversible context~2 in KAM theory refers to the situation where
$\dim\Fix G<\frac{1}{2}\codim\cT$, here $\Fix G$ is the fixed point manifold
of the reversing involution $G$ and $\cT$ is the invariant torus one deals
with. Up to now, the persistence of invariant tori in the reversible context~2
has been only explored in the extreme particular case where $\dim\Fix G=0$
[M.~B.~Sevryuk, Regul.\ Chaotic Dyn.\ \textbf{16} (2011), no.~1--2, 24--38].
We obtain a KAM-type result for the reversible context~2 in the general
situation where the dimension of $\Fix G$ is arbitrary. As in the case where
$\dim\Fix G=0$, the main technical tool is J.~Moser's modifying terms theorem
of 1967.
\end{abstract}

\maketitle

\baselineskip=17pt

\section{Introduction}\label{intro}

\subsection{Reversible Systems}\label{revers}

Any diffeomorphism $G:\cM\to\cM$ of a manifold $\cM$ induces the inner
automorphism $\Ad_G$ of the Lie algebra of vector fields on $\cM$ according to
the formula $\Ad_GV=TG(V\circ G^{-1})$. We will always assume the phase space
$\cM$ to be finite dimensional and connected.

\begin{defn}\label{defrev}
A vector field $V$ on $\cM$ is said to be \emph{equivariant with respect to
$G$} (or \emph{$G$-equivariant}) if $\Ad_GV=V$ and is said to be \emph{weakly
reversible with respect to $G$} (or \emph{weakly $G$-reversible}) if
$\Ad_GV=-V$. If the diffeomorphism $G$ is an \emph{involution} (i.e., its
square is the identity transformation), then the word ``\emph{weakly}'' is
usually omitted.
\end{defn}

For instance, the Newtonian equations of motion
$\ddot{\br}=\cF(\br,\dot{\br})$, $\br\in\mR^{\cK}$, are reversible with
respect to the phase space involution
$G:(\br,\dot{\br})\mapsto(\br,-\dot{\br})$ if and only if the forces $\cF$ are
even in the velocities $\dot{\br}$ (e.g., are independent of $\dot{\br}$).
This is probably the best known example of a reversible vector field. The
papers \cite{LR98,RQ92} present general surveys of the theory of finite
dimensional reversible systems with extensive bibliographies. There exists
a remarkable deep similarity between reversible and Hamiltonian dynamics
\cite{A84,AS86,LR98,RQ92,S86,S91}. In particular, a great deal of the
Hamiltonian KAM theory can be carried over to the reversible realm
\cite{AS86,BCHV09,BHN07,BH95,BHS96Gro,BHS96LNM,M67,M73,S86,S95Cha,S98,S06,%
S07DCDS,S07Stek}. Of the three great founders of KAM theory (A.~N.~Kolmogorov,
V.~I.~Arnold, and J.~K.~Moser, all in blessed remembrance now), the last
two---Arnold \cite{A84,AS86} and Moser \cite{M65,M67,M73}---made fundamental
contributions not only to the Hamiltonian version of this theory but also to
its reversible counterpart. The reversible KAM theory started with Moser's
paper~\cite{M65}. Arnold~\cite{A84} suggested that KAM theory for reversible
systems can be generalized to weakly reversible systems, and this idea was
soon realized \cite{AS86,S86}. A brief survey of the reversible KAM theory as
it stood in 1997 is presented in the review~\cite{S98}. Of subsequent works,
one may mention e.g.\ the references
\cite{BCHV09,BHN07,L01,S06,S07DCDS,S07Stek,S11,WX09,WXZ10,WZX11,W01,X04,Z08}.

\subsection{Invariant Tori of Reversible Flows}\label{tori}

The key object of KAM theory is an invariant torus carrying quasi-periodic
motions. In the case of weakly reversible systems, such a torus is always
assumed to be invariant not only under the dynamical system itself but also
under the corresponding reversing diffeomorphism \cite{AS86,S86,S98,S11}. So,
consider an $n$-torus $\cT$ ($n\geq 1$) invariant under both the flow of a
weakly reversible vector field and its reversing diffeomorphism $G$. Suppose
that $\cT$ carries \emph{conditionally periodic} motions, i.e., in some
angular coordinates $\varp\in\mT^n=(\mR/2\pi\mZ)^n$ in $\cT$, the dynamics on
$\cT$ takes the form $\dot{\varp}=\omega$ with some constant $\omega\in\mR^n$.

\begin{lem}[\cite{BHS96Gro,BHS96LNM,S86}]\label{lemNF}
If the motions on $\cT$ are \emph{quasi-periodic}, i.e., the components of
$\omega$ are rationally independent, then the diffeomorphism $G|_{\cT}$ has
the form $\varp\mapsto\Delta-\varp$ with some constant $\Delta\in\mT^n$ and is
therefore an involution. The coordinate shift $x=\varp-\Delta/2$ puts this
involution into the standard form $G|_{\cT}:x\mapsto-x$ while the dynamics is
still given by the equation $\dot{x}=\omega$.
\end{lem}

\begin{proof}
The function $t\mapsto\varp(t)=\omega t$ is a solution of the equation
$\dot{\varp}=\omega$. The function $t\mapsto G\bigl(\varp(-t)\bigr)$ is also a
solution (due to weak reversibility), so that $G(-\omega t)=G(0)+\omega t$ for
any $t\in\mR$. Since the orbit $\{\omega t \mid t\in\mR\}$ is dense in
$\mT^n$, we conclude that $G(\varp)=G(0)-\varp$ for any $\varp\in\mT^n$.
\end{proof}

The incommensurability condition on $\omega_1,\ldots,\omega_n$ in
Lemma~\ref{lemNF} is essential. For instance, the zero vector field
($\omega=0$) on a torus is reversed by any diffeomorphism.

Despite its triviality, Lemma~\ref{lemNF} has three very important
consequences. First, it suggests that a passage from reversible systems to
weakly reversible ones cannot affect drastically KAM-type theorems (and this
is indeed so \cite{AS86,S86,S98}). Second, under the hypotheses of
Lemma~\ref{lemNF}, the fixed point set $\Fix(G|_{\cT})$ is a collection of
$2^n=2^{\dim\cT}$ isolated points:
\begin{equation}
\Fix(G|_{\cT}) = (\Fix G)\cap\cT = \{0;\pi\}^n =
\{x\in\mT^n \mid \text{$x_\ell=0$ or $x_\ell=\pi$}, \; 1\leq\ell\leq n\}.
\label{eqFix}
\end{equation}
Third, if the fixed point set $\Fix G$ is a smooth submanifold of the phase
space $\cM$ and all the connected components of $\Fix G$ are of the same
dimension (so that $\dim\Fix G$ is well defined), then
$\dim\Fix G\leq\codim\cT$ ($\codim\cT$ being the codimension of $\cT$ in
$\cM$). Indeed, at each point $\slw\in(\Fix G)\cap\cT$, the number $-1$ is an
eigenvalue of $T_{\slw}G$ of multiplicity at least $n=\dim\cT$.

If $G:\cM\to\cM$ is an involution, then
$\Fix G=\{\slw\in\cM \mid G(\slw)=\slw\}$ is always a smooth submanifold of
$\cM$ of the same smoothness class as $G$ itself (a particular case of the
Bochner theorem \cite{Br72,MZ74}). This submanifold can well be empty or
consist of several connected components of different dimensions. Extensive
information on the fixed point submanifolds of involutions of various
manifolds is presented in the books \cite{Br72,CF64}, see also the articles
\cite{QS93,S95RMS,S11}. In the sequel, while speaking of reversing involutions
(or, as they are sometimes called, \emph{reversing symmetries} or
\emph{reversers}) $G$, we will always suppose that $\Fix G\neq\varnothing$ and
all the connected components of $\Fix G$ are of the same dimension, so that
$\dim\Fix G$ is well defined (this is the case for almost all the reversible
systems encountered in practice).

The ubiquity of invariant tori carrying conditionally periodic motions stems,
in the long run, from the fact that any finite-dimensional connected compact
Abelian Lie group is a torus.

\subsection{Reversible Contexts~1 and~2}\label{context}

Let a $G$-reversible system admit an invariant $n$-torus $\cT$ carrying
quasi-periodic motions. As we saw above, $(\Fix G)\cap\cT$ is a collection of
$2^n$ points and $\dim\Fix G\leq\codim\cT$.

\begin{defn}[\cite{BHS96Gro,BHS96LNM,S08,S11}]\label{defcont}
The situation where $\frac{1}{2}\codim\cT\leq\dim\Fix G\leq\codim\cT$ is
called the \emph{reversible context~1}, whereas the opposite situation where
$0\leq\dim\Fix G<\frac{1}{2}\codim\cT$ is called the \emph{reversible
context~2}.
\end{defn}

An extensive discussion on the differences between reversible contexts~1 and~2
is presented in~\cite{S11} and will not be repeated here (the paper~\cite{S11}
also contains an example where the reversible context~2 appears quite
naturally). We will confine ourselves with some brief observations concerning
integrable cases of the \emph{extreme} reversible contexts~1 and~2.

Let the phase space variables be $x\in\mT^n$ and $y\in\cY\subset\mR^m$ where
$\cY$ is an open and connected domain ($n\geq 1$, $m\geq 1$). On the phase
space $\mT^n\times\cY$, consider vector fields reversible with respect to one
of the two involutions $G:(x,y)\mapsto(-x,\delta y)$, $\delta=\pm 1$ (in the
case of $\delta=-1$, the domain $\cY$ is assumed to contain the origin $0$ and
to be symmetric with respect to the origin: $-y\in\cY$ whenever $y\in\cY$).
Suppose that these vector fields are integrable [equivariant with respect to
the torus translations $(x,y)\mapsto(x+\const,y)$] and depend on a parameter
$\nu\in\cN\subset\mR^s$ where $\cN$ is an open and connected domain
($s\geq 0$). We are interested in invariant $n$-tori of such reversible
systems of the form $\cT=\{y=\const\}$.

If $\delta=1$ then $\Fix G=\{0;\pi\}^n\times\cY$ [cf.~\eqref{eqFix}], and we
have the \emph{extreme reversible context~1}: $\dim\Fix G=m=\codim\cT$. One
easily sees that integrable $G$-reversible systems in this case have the form
\begin{equation}
\dot{x}=H(y,\nu), \quad \dot{y}=0
\label{eqextreme1}
\end{equation}
with an arbitrary function $H$. For any value of $\nu$, the phase space is
foliated into $n$-tori $\{y=\const\}$ invariant under both the flow itself and
the reversing involution $G$ and carrying conditionally periodic motions with
frequency vectors $H(y,\nu)$. This picture is very similar to the dynamics of
completely integrable Hamiltonian systems (in the Liouville--Arnold
sense~\cite{AKN06}), $(y,x)$ playing the roles of action-angle variables.
Under suitable nondegeneracy conditions on the unperturbed frequency map
$(y,\nu)\mapsto H(y,\nu)$, one can prove various KAM-type theorems for
systems~\eqref{eqextreme1}
\cite{AS86,BH95,BHS96Gro,BHS96LNM,M67,M73,S86,S95Cha,S98,S06}.

If $\delta=-1$ then $\Fix G=\{0;\pi\}^n\times\{0\}$ [cf.~\eqref{eqFix}], and
we have the \emph{extreme reversible context~2}: $\dim\Fix G=0$ (in fact,
$\Fix G$ consists of $2^n$ isolated points). It is easy to verify that
integrable $G$-reversible systems in this case have the form
\begin{equation}
\dot{x}=H(y,\nu), \quad \dot{y}=P(y,\nu)
\label{eqextreme2}
\end{equation}
with arbitrary functions $H$ and $P$ \emph{even} in $y$. The only $n$-torus of
the form $\{y=\const\}$ invariant under the reversing involution $G$ is
$\{y=0\}$, and this torus is invariant under the flow of~\eqref{eqextreme2} if
and only if $P(0,\nu)=0$. Generically, the latter equation has no solutions
for $s<m$ and determines an $(s-m)$-dimensional surface in $\cN$ for
$s\geq m$.

Thus, the extreme reversible contexts~1 and~2 are drastically
different~\cite{S11}. In the extreme reversible context~2, the setup is not
local with respect to $y$ (and $y=0$ is a special value), $y$ is not an analog
of the action variable (one may call $y$ an \emph{anti-action} variable), and
even integrable systems with less than $m$ external parameters
(e.g.\ individual integrable systems) have generically no invariant tori.

By now, the reversible KAM theory in context~1 is nearly as developed as the
Hamiltonian KAM theory \cite{A84,AS86,B91,BCHV09,BHN07,BH95,BHS96Gro,%
BHS96LNM,L01,M65,M66,M67,M73,QS93,Sch87,S86,S91,S95RMS,S95Cha,S98,S06,%
S07DCDS,S07Stek,WX09,WXZ10,WZX11,W01,X04,Z08}. On the other hand, to the best
of my knowledge, the first and only result on the reversible context~2 was
published no earlier than in 2011~\cite{S11}. In the paper~\cite{S11}, we
studied the \emph{extreme} reversible context~2 ($\dim\Fix G=0$) and proved a
KAM-type theorem for small $G$-reversible perturbations of
systems~\eqref{eqextreme2} for sufficiently many (at least $n+m$) external
parameters. The main result of~\cite{S11} was obtained as an almost immediate
corollary of J.~Moser's modifying terms theory~\cite{M67} of 1967. At the end
of~\cite{S11}, we listed ten tentative topics and directions for further
research. The fourth topic was the non-extreme reversible context~2
($\dim\Fix G>0$).

The goal of the present paper is to prove a first theorem in the general
reversible context~2 for an \emph{arbitrary} dimension of $\Fix G$ (provided
that $\dim\Fix G<\frac{1}{2}\codim\cT$). As the principal technical tool, we
again use Moser's modifying terms theory~\cite{M67} which turns out to be very
powerful.

Invariant tori $\cT$ in the non-extreme reversible context~1
($\frac{1}{2}\codim\cT\leq\dim\Fix G<\codim\cT$) are often said to be lower
dimensional \cite{B91,BCHV09,BHN07,BH95,BHS96Gro,BHS96LNM,L01,Sch87,S91,%
S95Cha,S98,S06,S07DCDS,S07Stek,WX09,WXZ10,WZX11,W01,X04,Z08}. Similarly, in
the Hamiltonian KAM theory, isotropic invariant tori whose dimension is less
than the number of degrees of freedom (for general references, see
e.g.\ \cite{AKN06,BHS96LNM,BHTB90,BS10}) are also said to be lower
dimensional. Following this pattern, we call invariant tori $\cT$ in the
non-extreme reversible context~2 ($0<\dim\Fix G<\frac{1}{2}\codim\cT$)
\emph{lower dimensional} as well.

The task of developing the reversible KAM theory in context~2 was listed
(as problem~9) among the ten problems of the classical KAM theory in the
note~\cite{S08}.

\subsection{Structure of the Paper}\label{struct}

The paper is organized as follows. A precise formulation of the particular
case of Moser's theorem we need is given in Section~\ref{Moser}. In
Section~\ref{we}, we explain why Theorem~\ref{thmain} of Section~\ref{Moser}
can be applied to reversible systems. Our main result for the non-extreme
reversible context~2 and its detailed proof are presented in
Section~\ref{cont2}. For comparison, the parallel result for the non-extreme
reversible context~1 is formulated in Section~\ref{cont1}.

\subsection{Dedication}\label{Arnold}

This paper is dedicated to the memory of V.~I.~Arnold, one of the most
creative, deep, distinctive, versatile, prolific, and influential scholars in
the history of mathematics. Under his supervision and with his generous help,
I mastered KAM theory and the theory of reversible systems in \mbox{1983--87}
as a student and post-graduate student at the Moscow State University, wrote
the book~\cite{S86}, and defended my PhD thesis ``Reversible dynamical
systems'' in 1988. Many subsequent important events in my life would not have
occurred without Arnold either. His untimely death is a bereavement for the
world scientific and educational community.

\section{Moser's Theorem}\label{Moser}

\subsection{Notation:\ Part~1}\label{nota1}

In this section, we present the particular case of Moser's modifying terms
theorem~\cite{M67} suitable for our purposes. In fact, the modifying terms
(in a somewhat different form) first appeared in the review~\cite{M66}. Some
generalizations of Moser's modifying terms theorem are obtained in the works
\cite{B73,B91,CGP11,Sch87,S11,Wa10}, see also a discussion in the memoir
\cite[Part~I,~\S\,7\textbf{b}]{BHTB90}.

In the sequel, $\cO_{\cK}(\slw)$ will denote an unspecified neighborhood
of a point $\slw\in\mR^{\cK}$. We will also use the notation
$\partial_{\slw}=\partial/\partial\slw$ for any variable $\slw$ and
$|\slw|=|\slw_1|+\cdots+|\slw_{\cK}|$ for any variable or vector
$\slw\in\mR^{\cK}$. The standard inner product of real vectors will be denoted
by $\langle\cdot,\cdot\rangle$ while the Poisson bracket of vector fields, by
$[\cdot,\cdot]$. For $\slw\in\mR$, we will employ the notation
$\lfloor\slw\rfloor=\max\{\ell\in\mZ \mid \ell\leq\slw\}$. For
$\slw\in\cO_{\cK}(0)$, we will write $O(\slw)$ instead of
$O\bigl(|\slw|\bigr)$ and $O_\ell(\slw)$ instead of $O\bigl(|\slw|^\ell\bigr)$
for $\ell\geq 2$. Similarly, $O(\slw',\slw'')$ will mean
$O\bigl(|\slw'|+|\slw''|\bigr)$ and $O_\ell(\slw',\slw'')$ will mean
$O\left(\bigl(|\slw'|+|\slw''|\bigr)^\ell\right)$ for $\ell\geq 2$. Finally,
$I_{\cK}$ and $0_{\cK\times\cL}$ will denote the identity $\cK\times\cK$
matrix and the zero $\cK\times\cL$ matrix, respectively.

The phase space variables in this section and the next one are $x\in\mT^n$ and
$X\in\cO_N(0)$, whereas $\nu\in\cN\subset\mR^s$ is an external parameter and
$\vare\in\cO_1(0)$ is the perturbation parameter ($n\geq 1$, $N\geq 0$,
$s\geq 0$), $\cN$ being an open and connected domain for $s\geq 1$. Boldface
letters will denote $N\times N$ matrices and matrix-valued functions. We will
keep most of Moser's original notation~\cite{M67}.

\subsection{Vector $\omega$ and Matrices $\bOmega(\nu)$}\label{omega}

Fix a vector $\omega\in\mR^n$ and an analytic matrix-valued function
$\bOmega:\cN\to\gl(N,\mR)$ satisfying the following hypothesis.

\begin{cond}[Diagonalizability, imaginary constancy, and Diophantine property]
The matrix $\bOmega(\nu)$ is diagonalizable over $\mC$ for any $\nu\in\cN$ and
the multiplicities of its eigenvalues do not depend on $\nu$. Moreover, the
\emph{imaginary parts} of the eigenvalues of $\bOmega(\nu)$ do \emph{not}
depend on $\nu$ either. If $\beta_1,\ldots,\beta_d$
($0\leq d\leq\lfloor N/2\rfloor$) are the positive imaginary parts (counting
multiplicities) of the eigenvalues of $\bOmega(\nu)$, then there exist numbers
$\tau>n-1$ and $\gamma>0$ such that
\begin{equation}
\bigl| \langle j,\omega\rangle+\langle J,\beta\rangle \bigr| \geq
\gamma|j|^{-\tau}
\label{eqDioph}
\end{equation}
for all $j\in\mZ^n\setminus\{0\}$ and $J\in\mZ^d$, $|J|\leq 2$.
\end{cond}

The inequalities~\eqref{eqDioph} imply, in particular, that $\omega$ is
Diophantine in the usual sense. Clearly, these inequalities (and even
inequalities $\langle j,\omega\rangle+\langle J,\beta\rangle\neq 0$) cannot be
valid for all $\nu\in\cN$ in the case of a non-constant continuous function
$\beta:\cN\to\mR^d$ (provided that $n\geq 2$) because the set
$\bigl\{ \langle j,\omega\rangle \bigm| j\in\mZ^n \bigr\}$ is dense in $\mR$
for rationally independent $\omega_1,\ldots,\omega_n$.

\subsection{Sets of Vector Fields and Diffeomorphisms}\label{LG}

Now introduce the following spaces of analytic vector fields and groups of
analytic diffeomorphisms (their meaning is explained in Moser's
paper~\cite{M67} and in Subsection~\ref{setup} below):
\begin{itemize}
\item $\fL$, the Lie algebra of all the analytic vector fields on
$\mT^n\times\cO_N(0)$;
\item $\fG$, the corresponding Lie group of analytic diffeomorphisms;
\item $\fL_1\subset\fL$, the Lie subalgebra of all the vector fields of the
form
\[
a(x)\partial_x+\bigl[b(x)+\bc(x)X\bigr]\partial_X
\]
(it is very easy to verify that the space of such vector fields is indeed
closed with respect to the Poisson bracket);
\item $\fG_1\subset\fG$, the corresponding Lie subgroup of analytic
diffeomorphisms of the form
\[
(x,X)\mapsto\bigl(\cA(x), \, \cB(x)+\bcC(x)X\bigr)
\]
where $\cA$ is a diffeomorphism of $\mT^n$ and $\det\bcC(x)\neq 0$ for any
$x\in\mT^n$;
\item $\ZfL\subset\fL$, a certain Lie subalgebra of vector fields;
\item $\ZfG\subset\fG$, the corresponding Lie subgroup of diffeomorphisms;
\item $\fM\subset\fL$, a certain vector \emph{subspace} of $\fL$ (not
necessarily a Lie subalgebra).
\end{itemize}

The vector fields
\[
D_\nu=\omega\partial_x+\bOmega(\nu)X\partial_X\in\fL_1
\]
will be treated as the unperturbed ones. Condition~$\Omega$ implies that the
adjoint operators
\[
\vart_\nu=\ad D_\nu:\fL_1\to\fL_1, \quad \vart_\nu V=[D_\nu,V]
\]
are semisimple \cite[Section~2b)]{M67}: for each $\nu\in\cN$, the algebra
$\fL_1$ is decomposed as
\begin{equation}
\fL_1=\fR_\nu\oplus\cR_\nu, \qquad \fR_\nu=\Ker\vart_\nu, \quad
\cR_\nu=\Image\vart_\nu=\vart_\nu(\fL_1),
\label{eqfRcR}
\end{equation}
where the nullspace $\fR_\nu$ of $\vart_\nu$ is finite dimensional
\cite[Section~2b)]{M67}. Moreover, this decomposition depends analytically on
$\nu$, and so does the operator $\vart_\nu^{-1}:\cR_\nu\to\cR_\nu$. In
particular, $\dim\fR_\nu$ does not depend on $\nu$. One straightforwardly
verifies that
\begin{align*}
\vart_\nu\bigl( a\partial_x+{} & (b+\bc X)\partial_X \bigr) = {} \\
& a_x\omega\partial_x+\bigl[ b_x\omega-\bOmega(\nu)b +
\bigl(\bc_x\omega+\bc\bOmega(\nu)-\bOmega(\nu)\bc\bigr)X \bigr]\partial_X,
\end{align*}
where $a=a(x)$,
$a_x\omega=\omega_1\partial_{x_1}a+\cdots+\omega_n\partial_{x_n}a$, and the
same notation is used for $b$ and $\bc$. Invoking Condition~$\Omega$ again, we
conclude that \cite[Section~2b)]{M67}
\begin{equation}
\begin{aligned}
\fR_\nu=\bigl\{ a_0\partial_x+(b_0+\bc_0X)\partial_X \bigm| {} &
\text{$a_0\in\mR^n$, $b_0\in\mR^N$, and $\bc_0\in\gl(N,\mR)$ are constants} \\
& \text{such that $\bOmega(\nu)b_0=0$ and
$\bOmega(\nu)\bc_0=\bc_0\bOmega(\nu)$} \bigr\}.
\end{aligned}
\label{eqKer}
\end{equation}

The Lie algebra $\ZfL$, Lie group $\ZfG$, space $\fM$, and vector fields
$D_\nu$ are assumed to satisfy the following requirements (cf.~\cite{S11}):
\begin{enumerate}
\item\label{invariant} The space $\fM$ is invariant under the action of $\ZfG$
by inner automorphisms: $\Ad_WV\in\fM$ whenever $W\in\ZfG$ and $V\in\fM$.
\item\label{linear} The spaces $\ZfL$ and $\fM$ are closed with respect to
``linearizations'': if
\[
\slu(x,X)\partial_x+\slv(x,X)\partial_X\in\ZfL \quad (\text{or}\in\fM),
\]
then
\[
\slu(x,0)\partial_x+
\left[\slv(x,0)+\frac{\partial\slv(x,0)}{\partial X}X\right]\partial_X\in\ZfL
\quad (\text{respectively}\in\fM).
\]
\item\label{D} $D_\nu\in\fM$ for any $\nu\in\cN$.
\item\label{intersect} The decomposition $\fL_1=\fR_\nu\oplus\cR_\nu$
``persists'' under the intersection with $\fM$ for any $\nu\in\cN$, i.e.,
\[
\fL_1\cap\fM=(\fR_\nu\cap\fM)\oplus(\cR_\nu\cap\fM),
\]
and both the spaces $\fR_\nu\cap\fM$ and $\cR_\nu\cap\fM$ depend analytically
on $\nu$.
\item\label{restore} $V\in\ZfL$ whenever $\nu\in\cN$, $V\in\cR_\nu$, and
$\vart_\nu V\in(\cR_\nu\cap\fM)$.
\end{enumerate}

\subsection{Formulation of the Theorem}\label{form}

Under Condition~$\Omega$ and conditions \mbox{\ref{invariant}--\ref{restore}}
just presented, the following statement holds.

\begin{thm}[Moser {\cite[Section~5d)]{M67}}, see also a discussion
in~\cite{S11}]\label{thmain}
Consider a family of systems of differential equations
\[
\dot{x}=\omega+\vare f(x,X,\nu,\vare), \quad
\dot{X}=\bOmega(\nu)X+\vare F(x,X,\nu,\vare),
\]
where the perturbation terms $f$ and $F$ are analytic in all their arguments
and
\[
f(x,X,\nu,\vare)\partial_x+F(x,X,\nu,\vare)\partial_X \in \fM
\]
for any $\nu$ and $\vare$. Then, for $\vare$ sufficiently small, there exist
unique analytic functions
\begin{equation}
\begin{aligned}
\lambda:(\nu,\vare) &\mapsto \lambda(\nu,\vare)\in\mR^n, \\
\mu:(\nu,\vare) &\mapsto \mu(\nu,\vare)\in\mR^N, \\
\bM:(\nu,\vare) &\mapsto \bM(\nu,\vare)\in\gl(N,\mR)
\end{aligned}
\label{eqmod}
\end{equation}
such that
\begin{equation}
\lambda(\nu,\vare)\partial_x +
\bigl[\mu(\nu,\vare)+\bM(\nu,\vare)X\bigr]\partial_X \in (\fR_\nu\cap\fM)
\label{eqlmbM}
\end{equation}
for any $\nu$ and $\vare$ and the following holds. For any $\nu$ and any
sufficiently small $\vare$, there is a coordinate transformation in
$\fG_1\cap\ZfG$ of the form
\begin{equation}
x=\xi+\vare A(\xi,\nu,\vare), \quad
X=\Xi+\vare B(\xi,\nu,\vare)+\vare\bC(\xi,\nu,\vare)\Xi
\label{eqtrans}
\end{equation}
\textup{[}\,$\xi\in\mT^n$ and $\Xi\in\cO_N(0)$ being new phase space
variables\/\textup{]} that casts the \emph{modified} system
\begin{equation}
\begin{aligned}
\dot{x} &= \omega+\vare f(x,X,\nu,\vare) + \vare\lambda(\nu,\vare), \\
\dot{X} &= \bOmega(\nu)X+\vare F(x,X,\nu,\vare) +
\vare\mu(\nu,\vare)+\vare\bM(\nu,\vare)X
\end{aligned}
\label{eqnews}
\end{equation}
into a system of the form
\begin{equation}
\dot{\xi}=\omega+\vare O(\Xi), \quad
\dot{\Xi}=\bOmega(\nu)\Xi+\vare O_2(\Xi).
\label{eqclear}
\end{equation}
The coefficients $A$, $B$, and $\bC$ in~\eqref{eqtrans} are analytic in $\xi$,
$\nu$, and $\vare$.
\end{thm}

The required smallness of $\vare$ in Theorem~\ref{thmain} is uniform in $\nu$
ranging in any fixed compact subset of $\cN$. The functions~\eqref{eqmod} are
usually called Moser's \emph{modifying terms}. For any $\nu$ and $\vare$, the
\emph{modifying vector field}~\eqref{eqlmbM} lies in the \emph{finite
dimensional} space $\fR_\nu\cap\fM$.

Generally speaking, a coordinate transformation~\eqref{eqtrans} is not unique.
Suppose that a transformation
\[
\xi\new=\xi+\vare\Delta(\nu,\vare), \quad \Xi\new=\Xi+\vare\bS(\nu,\vare)\Xi
\]
(with analytic coefficients $\Delta$ and $\bS$) lies in $\ZfG$. This
additional coordinate transformation retains the form~\eqref{eqclear} of the
system~\eqref{eqnews} provided that the matrix $I_N+\vare\bS(\nu,\vare)$
commutes with $\bOmega(\nu)$ for any $\nu$ and $\vare$.

\subsection{Reducible Invariant Tori}\label{reduce}

In the lower dimensional KAM theory, the following concepts are of principal
importance.

\begin{defn}[\cite{AKN06,BHS96Gro,BHS96LNM,M66,M67,S98}]\label{defFloq}
Let an invariant $n$-torus $\cT$ of some flow on an $(n+N)$-dimensional
manifold carry conditionally periodic motions with frequency vector
$\omega\in\mR^n$. This torus is said to be \emph{reducible} (or
\emph{Floquet}) if in a neighborhood of $\cT$, there exist coordinates
$\bigl(x\in\mT^n, \, X\in\cO_N(0)\bigr)$ in which the torus $\cT$ itself is
given by the equation $\{X=0\}$ and the dynamical system takes the
\emph{Floquet form} $\dot{x}=\omega+O(X)$, $\dot{X}=\bOmega X+O_2(X)$ with an
$x$-independent matrix $\bOmega\in\gl(N,\mR)$. This matrix (not determined
uniquely) is called the \emph{Floquet matrix} of the torus $\cT$, and its
eigenvalues are called the \emph{Floquet exponents} of $\cT$.
\end{defn}

In other words, an invariant torus is reducible if the variational equation
along this torus can be reduced to a form with constant coefficients. One sees
that the modified system~\eqref{eqnews} in Theorem~\ref{thmain} possesses a
reducible invariant $n$-torus $\{\Xi=0\}$ with frequency vector $\omega$ and
Floquet matrix $\bOmega(\nu)$.

The version of Moser's modifying terms theorem we have just presented differs
from the version in our previous paper~\cite{S11} in the only point: in
Theorem~\ref{thmain} above, the matrix $\bOmega$ is allowed to depend on the
external parameter $\nu$. In fact, in Moser's original paper~\cite{M67}, all
versions of the modifying terms theorem were given without an external
parameter whatever. But Moser himself pointed out \cite[p.~170]{M67} that an
extension to differential equations depending analytically on several
parameters is obvious. While examining lower dimensional invariant tori with
real Floquet exponents in Hamiltonian systems \cite[Section~6d)]{M67}, Moser
treated positive Floquet exponents as external parameters, so that the
corresponding Floquet matrix turned out to be parameter-dependent. He wrote
\cite[p.~174]{M67}: ``We do not have to keep the $\Omega_\mu$ [the eigenvalues
of $\bOmega$] fixed but only their imaginary part. Then the real parts of
these eigenvalues can be considered as additional parameters.''

\section{Applications to Reversible Systems}\label{we}

\subsection{Preliminaries}\label{prelim}

Before proceeding to applications of Moser's Theorem~\ref{thmain}, recall some
simple general facts about the action of diffeomorphisms on vector fields.
Consider an arbitrary diffeomorphism $G:\cM\to\cM$ of a manifold $\cM$.

\begin{lem}\label{lemA}
If $\Ad_GV=\delta V$ \textup{(}$\delta=\pm 1$\textup{)}, then
$\Ad_G[V,U]=\delta[V,\Ad_GU]$ for any vector field $U$. If $\Ad_GV=\delta_1V$
and $\Ad_GU=\delta_2U$ \textup{(}$\delta_1=\pm 1$, $\delta_2=\pm 1$\textup{)},
then $\Ad_G[V,U]=\delta_1\delta_2[V,U]$.
\end{lem}

\begin{proof}
Indeed, $\Ad_G$ is an automorphism of the Lie algebra of vector fields on
$\cM$.
\end{proof}

\begin{lem}\label{lemB}
If a diffeomorphism $W:\cM\to\cM$ commutes with $G$, then $\Ad_W$ casts
$G$-equivariant vector fields to $G$-equivariant ones and weakly
$G$-reversible vector fields to weakly $G$-reversible ones.
\end{lem}

\begin{proof}
The correspondence $G\mapsto\Ad_G$ is a representation of the group of
diffeomorphisms of $\cM$. Hence, if $\Ad_GV=\delta V$ ($\delta=\pm 1$), then
$\Ad_G\Ad_WV=\Ad_{GW}V=\Ad_{WG}V=\Ad_W\Ad_GV=\delta\Ad_WV$.
\end{proof}

Let $\fH$ be a vector subspace of vector fields on $\cM$. If $\fH$ is
invariant under $\Ad_G$, we will write
\[
\fH^{+G}=\{V\in\fH \mid \Ad_GV=V\}, \quad \fH^{-G}=\{V\in\fH \mid \Ad_GV=-V\}.
\]

\begin{lem}\label{lemC}
If $G$ is an involution, then $\fH=\fH^{+G}\oplus\fH^{-G}$.
\end{lem}

\begin{proof}
For any $V\in\fH$, we have the decomposition
$V=\frac{1}{2}(V+\Ad_GV)+\frac{1}{2}(V-\Ad_GV)$. Since $G^2$ is the identity
transformation, $\frac{1}{2}(V+\Ad_GV)\in\fH^{+G}$ and
$\frac{1}{2}(V-\Ad_GV)\in\fH^{-G}$.
\end{proof}

The involutivity condition on $G$ in Lemma~\ref{lemC} is essential. For
instance, let $\cM=\mR$ and let $G$ be linear: $G(\slw)=\sla\slw$,
$\sla\neq 0$. Set $\fH$ to be the space of all smooth vector fields on $\mR$.
If $|\sla|\neq 1$ (so that $G$ is not an involution), then $\fH^{+G}$ is the
subspace of \emph{linear} vector fields $\slb\slw\partial_{\slw}$ and
$\fH^{-G}=\{0\}$. Thus, $\fH^{+G}\oplus\fH^{-G}$ is far from coinciding with
$\fH$.

\subsection{The Main Setup}\label{setup}

In the notation of Section~\ref{Moser}, consider an involution
$G:(x,X)\mapsto(-x,\bL X)$, where $\bL\in\GL(N,\mR)$ is an involutive matrix
($\bL^2=I_N$) and the neighborhood $\cO_N(0)\ni X$ is invariant under the
linear involution $X\mapsto\bL X$. The action of $\Ad_G$ on vector fields in
$\fL$ and $\fL_1$ is described by the formulas
\begin{gather}
\Ad_G\bigl( \slu(x,X)\partial_x+\slv(x,X)\partial_X \bigr) =
-\slu(-x,\bL X)\partial_x+\bL\slv(-x,\bL X)\partial_X,
\label{eqbig} \\
\Ad_G\bigl( a(x)\partial_x+\bigl[b(x)+\bc(x)X\bigr]\partial_X \bigr) =
-a(-x)\partial_x+\bigl[\bL b(-x)+\bL\bc(-x)\bL X\bigr]\partial_X.
\label{eqsmall}
\end{gather}
Consequently, the subalgebra $\fL_1$ is invariant under $\Ad_G$. According to
Lemma~\ref{lemC}, $\fL_1=\fL_1^{+G}\oplus\fL_1^{-G}$.

Given a vector $\omega\in\mR^n$ and an analytic matrix-valued function
$\bOmega:\cN\to\gl(N,\mR)$ satisfying Condition~$\Omega$ of
Subsection~\ref{omega} and such that $\bOmega(\nu)\bL\equiv-\bL\bOmega(\nu)$,
we will apply Theorem~\ref{thmain} to the situations where
\begin{itemize}
\item $\ZfL=\fL^{+G}\subset\fL$ is the Lie algebra of $G$-equivariant vector
fields;
\item $\ZfG\subset\fG$ is the corresponding Lie group of diffeomorphisms
commuting with $G$;
\item $\fM=\fL^{-G}\subset\fL$ is the space of vector fields reversible with
respect to $G$.
\end{itemize}
Note that $\fM$ here is \emph{not} a Lie algebra: according to
Lemma~\ref{lemA}, the Poisson bracket of two $G$-reversible vector fields
is a $G$-equivariant vector field rather than a $G$-reversible one.

The following result seems to be new. In the papers \cite{M67,S11}, similar
statements were verified for $\bL=\pm I_N$ only. The paper~\cite{Sch87}
considered only the case where the dimension of the $(-1)$-eigenspace of $\bL$
does not exceed $\lfloor N/2\rfloor$ which essentially corresponds to the
reversible context~1.

\begin{lem}\label{lemcent}
For \emph{any} involutive matrix $\bL$ and the choice of $\omega$,
$\bOmega(\nu)$, $\ZfL$, $\ZfG$, and $\fM$ just described, all the conditions
\textup{\mbox{\ref{invariant}--\ref{restore}}} of Subsection~\ref{LG} are
satisfied.
\end{lem}

\begin{proof}
Property~\ref{invariant} follows immediately from Lemma~\ref{lemB}.
Property~\ref{linear} can be verified by a very easy computation
using~\eqref{eqbig}. Property~\ref{D} is obvious because the matrices
$\bOmega(\nu)$ and $\bL$ anti-commute for any $\nu\in\cN$, and therefore
$D_\nu\in\fL_1^{-G}$.

Since $\Ad_GD_\nu=-D_\nu$, Lemma~\ref{lemA} implies that
$\vart_\nu(\Ad_GV)=-\Ad_G(\vart_\nu V)$ for any $\nu\in\cN$ and $V\in\fL$.
Consequently, both the spaces $\fR_\nu$ and $\cR_\nu$ in the
decomposition~\eqref{eqfRcR} are invariant under $\Ad_G$ for any $\nu\in\cN$.
According to Lemma~\ref{lemC}, $\fR_\nu=\fR_\nu^{+G}\oplus\fR_\nu^{-G}$ with
$\fR_\nu^{+G}=\fR_\nu\cap\ZfL$, $\fR_\nu^{-G}=\fR_\nu\cap\fM$ and
$\cR_\nu=\cR_\nu^{+G}\oplus\cR_\nu^{-G}$ with $\cR_\nu^{+G}=\cR_\nu\cap\ZfL$,
$\cR_\nu^{-G}=\cR_\nu\cap\fM$. Thus,
\[
\fL_1=\fR_\nu\oplus\cR_\nu =
\fR_\nu^{+G}\oplus\fR_\nu^{-G} \oplus \cR_\nu^{+G}\oplus\cR_\nu^{-G},
\]
and $\fL_1\cap\ZfL=\fL_1^{+G}=\fR_\nu^{+G}\oplus\cR_\nu^{+G}$,
$\fL_1\cap\fM=\fL_1^{-G}=\fR_\nu^{-G}\oplus\cR_\nu^{-G}$. Each of the four
subspaces $\fR_\nu^{+G}$, $\fR_\nu^{-G}$, $\cR_\nu^{+G}$, and $\cR_\nu^{-G}$
depends analytically on $\nu$ because the spaces $\fR_\nu$ and $\cR_\nu$
depend analytically on $\nu$ and are invariant under $\Ad_G$. We have verified
property~\ref{intersect}.

Finally, $\vart_\nu(\cR_\nu^{+G})=\cR_\nu^{-G}$ and
$\vart_\nu(\cR_\nu^{-G})=\cR_\nu^{+G}$ for any $\nu\in\cN$ according to
Lemma~\ref{lemA}. Consequently, if $V\in\cR_\nu$ for some $\nu$ and
$\vart_\nu V\in\cR_\nu^{-G}$, then $V\in\cR_\nu^{+G}$ (similarly, if
$\vart_\nu V\in\cR_\nu^{+G}$, then $V\in\cR_\nu^{-G}$). Thus,
property~\ref{restore} is also valid.
\end{proof}

Taking~\eqref{eqKer} and~\eqref{eqsmall} into account, we conclude that for
$\fM=\fL^{-G}$, the spaces $\fR_\nu\cap\fM=\fR_\nu^{-G}$ where Moser's
modifying vector fields~\eqref{eqlmbM} lie are
\begin{equation}
\begin{aligned}
\fR_\nu^{-G}=\bigl\{ & a_0\partial_x+(b_0+\bc_0X)\partial_X \bigm| {}
\text{$a_0\in\mR^n$, $b_0\in\mR^N$, and $\bc_0\in\gl(N,\mR)$ are constants} \\
& \text{such that $\bOmega(\nu)b_0=0$, $\bL b_0=-b_0$,
$\bOmega(\nu)\bc_0=\bc_0\bOmega(\nu)$, and $\bc_0\bL=-\bL\bc_0$} \bigr\}.
\end{aligned}
\label{eqkernel}
\end{equation}

\subsection{Information from Linear Algebra}\label{linalg}

In the rest of this paper, we will explore the case where
\begin{itemize}
\item $N=m+2p$ ($m\geq 0$, $p\geq 1$),
\item the matrix $\bL$ is block diagonal with blocks $\delta I_m$ and $K$,
where $\delta=\pm 1$ and $K\in\GL(2p,\mR)$ is an involutive matrix
($K^2=I_{2p}$) with eigenvalues $1$ and $-1$ of multiplicity $p$ each,
\item the matrices $\bOmega(\nu)$ are block diagonal for any $\nu\in\cN$ with
blocks $0_{m\times m}$ and $\Lambda(\nu)$, where the spectrum of the matrix
$\Lambda(\nu)\in\gl(2p,\mR)$ is simple for any $\nu\in\cN$ and
$\Lambda(\nu)K\equiv-K\Lambda(\nu)$.
\end{itemize}

The equality $\Lambda(\nu)K=-K\Lambda(\nu)$ implies that the eigenvalues of
$\Lambda(\nu)$ come in pairs $(\slw,-\slw)$, $\slw\in\mC$
\cite{H96,S86,S92,Sh93}. Since the spectrum of $\Lambda(\nu)$ is simple, it
does not contain $0$ and consists of real pairs $(\alpha,-\alpha)$, purely
imaginary pairs $(i\beta,-i\beta)$, and complex quadruplets
$\pm\alpha\pm i\beta$. In the sequel, we will use the following notation.

\begin{nott}
Let $\alpha_1,\ldots,\alpha_{d_1+d_3}$ and $\beta_1,\ldots,\beta_{d_2+d_3}$ be
arbitrary real numbers ($d_1$, $d_2$, and $d_3$ being non-negative integers
such that $d_1+d_2+2d_3=p$). Then
$\fT(d_1,d_2,d_3;\alpha,\beta)\in\gl(2p,\mR)$ denotes the block diagonal
matrix with the $d_1+d_2+d_3$ blocks
\begin{gather*}
\begin{pmatrix} 0 & \alpha_k \\ \alpha_k & 0 \end{pmatrix}, \quad
1\leq k\leq d_1, \qquad
\begin{pmatrix} 0 & \beta_l \\ -\beta_l & 0 \end{pmatrix}, \quad
1\leq l\leq d_2, \\
\begin{pmatrix}
0 & \alpha_{d_1+\iota} & 0 & \beta_{d_2+\iota} \\
\alpha_{d_1+\iota} & 0 & \beta_{d_2+\iota} & 0 \\
0 & -\beta_{d_2+\iota} & 0 & \alpha_{d_1+\iota} \\
-\beta_{d_2+\iota} & 0 & \alpha_{d_1+\iota} & 0
\end{pmatrix}, \quad 1\leq\iota\leq d_3.
\end{gather*}
\end{nott}

It is easy to see that the matrix $\fT(d_1,d_2,d_3;\alpha,\beta)$ is
diagonalizable over $\mC$ for any $\alpha\in\mR^{d_1+d_3}$ and
$\beta\in\mR^{d_2+d_3}$, and its eigenvalues are
\begin{equation}
\begin{gathered}
\pm\alpha_1,\ldots,\pm\alpha_{d_1}, \qquad
\pm i\beta_1,\ldots,\pm i\beta_{d_2}, \\
\pm\alpha_{d_1+1}\pm i\beta_{d_2+1},\ldots,
\pm\alpha_{d_1+d_3}\pm i\beta_{d_2+d_3}.
\end{gathered}
\label{eqspec}
\end{equation}

The theory of normal forms and versal unfoldings of matrices anti-commuting
with a fixed involutive matrix (of \emph{infinitesimally reversible} matrices
in the terminology of \cite{S86,S92}) has been developed in e.g.\ the papers
\cite{H96,S92,Sh93}. One of the simplest results of this theory is the
following statement.

\begin{lem}[\cite{H96,S92,Sh93}]\label{lemKLam}
Let linear operators $K$ and $\Lambda$ in $\gl(2p,\mR)$ satisfy the conditions
$K^2=I_{2p}$ and $\Lambda K=-K\Lambda$. Suppose that the spectrum of $\Lambda$
is simple and has the form~\eqref{eqspec} with $\alpha_k>0$
\textup{(}$1\leq k\leq d_1+d_3$\textup{)} and $\beta_l>0$
\textup{(}$1\leq l\leq d_2+d_3$\textup{)}. Then each of the two eigenvalues
$1$ and $-1$ of $K$ is of multiplicity $p$, and in a suitable basis of
$\mR^{2p}$, the matrices of $K$ and $\Lambda$ have the form
$K=\diag(1,-1,1,-1,\ldots,1,-1)$ and $\Lambda=\fT(d_1,d_2,d_3;\alpha,\beta)$.
Moreover, this basis can be chosen to depend analytically on $K$ and
$\Lambda$.
\end{lem}

Returning to the matrices
\[
\bL=\begin{pmatrix}
\delta I_m & 0_{m\times 2p} \\ 0_{2p\times m} & K
\end{pmatrix}, \qquad
\bOmega(\nu)=\begin{pmatrix}
0_{m\times m} & 0_{m\times 2p} \\ 0_{2p\times m} & \Lambda(\nu)
\end{pmatrix},
\]
one concludes from Lemma~\ref{lemKLam} that without loss of generality, we can
set
\[
K=\diag(1,-1,\ldots,1,-1), \qquad
\Lambda(\nu)=\fT\bigl(d_1,d_2,d_3;\alpha(\nu),\beta\bigr),
\]
where $d_1$, $d_2$, $d_3$ are non-negative integers independent of $\nu$ (and
such that $d_1+d_2+2d_3=p$), $\alpha_k(\nu)>0$ ($\nu\in\cN$,
$1\leq k\leq d_1+d_3$), and $\beta_l>0$ ($1\leq l\leq d_2+d_3$). Recall that
the imaginary parts of the eigenvalues of $\bOmega(\nu)$ do not depend on
$\nu$ (see Subsection~\ref{omega}).

\begin{lem}\label{lemfR}
For this choice of $\bL$ and $\bOmega(\nu)$, under the condition that the
spectrum of $\Lambda(\nu)$ is simple for any $\nu\in\cN$, the space
$\fR_\nu\cap\fM=\fR_\nu^{-G}\subset\fL_1$ for any $\nu\in\cN$ consists of
vector fields $a_0\partial_x+(b_0+\bc_0X)\partial_X$ where
\begin{itemize}
\item $a_0\in\mR^n$ are arbitrary,
\item $b_0=0$ for $\delta=1$ and
\[
b_0=\bigl(\, b_{01},\ldots,b_{0m},\underbrace{0,\ldots,0}_{2p} \,\bigr)
\]
with arbitrary real $b_{01},\ldots,b_{0m}$ for $\delta=-1$,
\item
\[
\bc_0=\begin{pmatrix}
0_{m\times m} & 0_{m\times 2p} \\
0_{2p\times m} & \fT(d_1,d_2,d_3;q,r)
\end{pmatrix}
\]
with arbitrary $q\in\mR^{d_1+d_3}$ and $r\in\mR^{d_2+d_3}$.
\end{itemize}
\end{lem}

\begin{proof}
This lemma easily follows from~\eqref{eqkernel}. The key point is to determine
the \emph{centralizer}
$\bigl\{\bc_0\in\gl(N,\mR) \bigm| \bOmega(\nu)\bc_0=\bc_0\bOmega(\nu)\bigr\}$
of $\bOmega(\nu)$. One can verify that this centralizer is the space of all
the block diagonal matrices $\bc_0$ with $1+d_1+d_2+d_3$ blocks of the form
\begin{gather*}
\Gamma\in\gl(m,\mR), \qquad
\begin{pmatrix} w_k & q_k \\ q_k & w_k \end{pmatrix}, \quad
1\leq k\leq d_1, \qquad
\begin{pmatrix} t_l & r_l \\ -r_l & t_l \end{pmatrix}, \quad
1\leq l\leq d_2, \\
\begin{pmatrix}
w_{d_1+\iota} & q_{d_1+\iota} & t_{d_2+\iota} & r_{d_2+\iota} \\
q_{d_1+\iota} & w_{d_1+\iota} & r_{d_2+\iota} & t_{d_2+\iota} \\
-t_{d_2+\iota} & -r_{d_2+\iota} & w_{d_1+\iota} & q_{d_1+\iota} \\
-r_{d_2+\iota} & -t_{d_2+\iota} & q_{d_1+\iota} & w_{d_1+\iota}
\end{pmatrix}, \quad 1\leq\iota\leq d_3.
\end{gather*}
For both values of $\delta$, the condition $\bc_0\bL=-\bL\bc_0$ is tantamount
to that $\Gamma=0$, $w_k=0$ ($1\leq k\leq d_1+d_3$), and $t_l=0$
($1\leq l\leq d_2+d_3$).
\end{proof}

The centralizer of $\bOmega(\nu)$ would be larger if the spectrum of
$\Lambda(\nu)$ were not simple.

\subsection{Notation:\ Part~2}\label{nota2}

From now on, we will consider reversible systems with phase space variables
$x\in\mT^n$, $y\in\cY\subset\mR^m$ ($\cY$ being an open domain), and
$z\in\cZ=\cO_{2p}(0)$ where $n\geq 1$, $m\geq 0$, and $p\geq 1$. These
reversible systems will be supposed to depend on an external parameter
$\nu\in\cN\subset\mR^s$ ($\cN$ being an open domain), $s\geq 0$, and on a
small perturbation parameter $\vare\geq 0$. The reversing involution will be
$G:(x,y,z)\mapsto(-x,\delta y,Kz)$ with $\delta=\pm 1$ and
$K=\diag(1,-1,\ldots,1,-1)$. The neighborhood $\cZ$ of $0\in\mR^{2p}$ is
assumed to be invariant under the linear involution $z\mapsto Kz$. Similarly,
for $\delta=-1$, the domain $\cY$ is assumed to contain the origin $0$ and to
be symmetric with respect to the origin.

We will look for reducible invariant $n$-tori $\cT$ (see
Definition~\ref{defFloq}) of such systems. Since $\codim\cT=m+2p$ and
\[
\text{$\dim\Fix G=m+p$ (for $\delta=1$) \quad or \quad
$\dim\Fix G=p$ (for $\delta=-1$)},
\]
we are within the non-extreme reversible context~1
($\frac{1}{2}\codim\cT\leq\dim\Fix G<\codim\cT$) for $\delta=1$ and within the
non-extreme reversible context~2 ($0<\dim\Fix G<\frac{1}{2}\codim\cT$) for
$\delta=-1$ and $m\geq 1$.

\section{The Reversible Context~2}\label{cont2}

Now we are in the position to formulate and prove the main result of this
paper. In the notation of Subsection~\ref{nota2}, let $m\geq 1$ and
$\delta=-1$, so that the reversing involution is $G:(x,y,z)\mapsto(-x,-y,Kz)$.
Consider a family of $G$-reversible systems on $\mT^n\times\cY\times\cZ$ of
the form
\begin{equation}
\begin{aligned}
\dot{x} &= H(y,\nu)+f^\sharp(x,y,z,\nu)+\vare f(x,y,z,\nu,\vare), \\
\dot{y} &= P(y,\nu)+g^\sharp(x,y,z,\nu)+\vare g(x,y,z,\nu,\vare), \\
\dot{z} &= Q(y,\nu)z+h^\sharp(x,y,z,\nu)+\vare h(x,y,z,\nu,\vare)
\end{aligned}
\label{eqsystem}
\end{equation}
(with $2p\times 2p$ matrix-valued function $Q$), where $f^\sharp=O(z)$,
$g^\sharp=O_2(z)$, and $h^\sharp=O_2(z)$ [cf.~\eqref{eqextreme2}].
Reversibility of~\eqref{eqsystem} with respect to $G$ means that
\begin{gather*}
H(-y,\nu)\equiv H(y,\nu), \quad P(-y,\nu)\equiv P(y,\nu), \\
Q(-y,\nu)K\equiv-KQ(y,\nu)
\end{gather*}
and
\begin{align*}
f^\sharp(-x,-y,Kz,\nu) &\equiv f^\sharp(x,y,z,\nu), &
f(-x,-y,Kz,\nu,\vare) &\equiv f(x,y,z,\nu,\vare), \\
g^\sharp(-x,-y,Kz,\nu) &\equiv g^\sharp(x,y,z,\nu), &
g(-x,-y,Kz,\nu,\vare) &\equiv g(x,y,z,\nu,\vare), \\
h^\sharp(-x,-y,Kz,\nu) &\equiv -Kh^\sharp(x,y,z,\nu), &
h(-x,-y,Kz,\nu,\vare) &\equiv -Kh(x,y,z,\nu,\vare).
\end{align*}
All the functions $H$, $P$, $Q$, $f^\sharp$, $g^\sharp$, $h^\sharp$, $f$, $g$,
and $h$ are assumed to be analytic in all their arguments.

Let the spectrum of the matrix $Q(0,\nu)$ anti-commuting with $K$ be simple
for any $\nu\in\cN$, and
\[
Q(0,\nu)=\fT\bigl(d_1,d_2,d_3;\alpha(\nu),\beta(\nu)\bigr),
\]
where the numbers $d_1\geq 0$, $d_2\geq 0$, $d_3\geq 0$ do not depend on $\nu$
($d_1+d_2+2d_3=p$), $\alpha_k(\nu)>0$ for all $1\leq k\leq d_1+d_3$,
$\nu\in\cN$, and $\beta_l(\nu)>0$ for all $1\leq l\leq d_2+d_3$, $\nu\in\cN$.
Introduce the notation $d_2+d_3=d$.

Fix an \emph{arbitrary} (possibly, empty) subset of indices
\[
\fZ\subset\{1;2;\ldots;d_1+d_3\}
\]
consisting of $\kappa$ elements ($0\leq\kappa\leq d_1+d_3$). We will write
\[
\alpha_+=(\alpha_k \mid k\in\fZ), \qquad \alpha_-=(\alpha_k \mid k\notin\fZ).
\]
Below, we will also use similar notation without special mention for vector
quantities in $\mR^{d_1+d_3}$ denoted by $\alpha^0$, $\alpha'$, $\alpha^\vee$,
$\alpha^\star$, and $q$.

\begin{thm}\label{thC2}
Suppose that
\begin{itemize}
\item $s\geq n+m+d+\kappa$,
\item $P(0,\nu^0)=0$ for some $\nu^0\in\cN$,
\item the vectors $\omega=H(0,\nu^0)\in\mR^n$ and
$\beta^0=\beta(\nu^0)\in\mR^d$ satisfy the following Diophantine condition:
there exist constants $\tau>n-1$ and $\gamma>0$ such that
\begin{equation}
\bigl| \langle j,\omega\rangle+\langle J,\beta^0\rangle \bigr| \geq
\gamma|j|^{-\tau}
\label{eqnewDioph}
\end{equation}
for all $j\in\mZ^n\setminus\{0\}$ and $J\in\mZ^d$, $|J|\leq 2$
\textup{[}cf.~\eqref{eqDioph}\textup{]},
\item the mapping
\[
\nu\mapsto\bigl( H(0,\nu), \, P(0,\nu), \, \beta(\nu), \, \alpha_+(\nu) \bigr)
\]
is submersive at point $\nu^0$, i.e.
\begin{equation}
\left. \rank\frac{
\partial\bigl( H(0,\nu), \, P(0,\nu), \, \beta(\nu), \, \alpha_+(\nu) \bigr)
}{\partial\nu} \right|_{\nu=\nu^0} = n+m+d+\kappa.
\label{eqrank}
\end{equation}
\end{itemize}
Then for sufficiently small $\vare$ there exists an
$(s-n-m-d-\kappa)$-dimensional analytic surface $\cS_\vare\subset\cN$ such
that for any $\nu\in\cS_\vare$, system~\eqref{eqsystem} admits an analytic
reducible invariant $n$-torus carrying quasi-periodic motions with frequency
vector $\omega$. The Floquet exponents of this torus are
\begin{equation}
\begin{gathered}
\underbrace{0,\ldots,0}_m\,, \qquad
\pm\alpha'_1(\nu,\vare),\ldots,\pm\alpha'_{d_1}(\nu,\vare), \qquad
\pm i\beta^0_1,\ldots,\pm i\beta^0_{d_2}, \\
\pm\alpha'_{d_1+1}(\nu,\vare)\pm i\beta^0_{d_2+1},\ldots,
\pm\alpha'_{d_1+d_3}(\nu,\vare)\pm i\beta^0_{d_2+d_3}
\end{gathered}
\label{eqprime}
\end{equation}
\textup{[}cf.~\eqref{eqspec}\textup{]}, where $\alpha'_k(\nu,\vare)>0$ for all
$1\leq k\leq d_1+d_3$, $\nu\in\cS_\vare$, and
$\alpha'_+(\nu,\vare)\equiv\alpha^0_+$, i.e.,
$\alpha'_k(\nu,\vare)\equiv\alpha_k(\nu^0)$ for $k\in\fZ$ \textup{[}here and
henceforth, $\alpha^0=\alpha(\nu^0)\in\mR^{d_1+d_3}$\textup{]}. These tori and
the numbers $\alpha'_k(\nu,\vare)$, $k\notin\fZ$, depend analytically on
$\nu\in\cS_\vare$ and on $\sqvare$. At $\vare=0$, the surface $\cS_0$ contains
$\nu^0$ and all the tori are $\{y=0, \, z=0\}$.
\end{thm}

\begin{proof}
As in the case of the extreme reversible context~2~\cite{S11}, to deduce
Theorem~\ref{thC2} from Theorem~\ref{thmain}, it in fact suffices to invoke
the implicit function theorem twice.

By virtue of~\eqref{eqrank}, one can introduce in $\cN$ near $\nu^0$ a new
coordinate system
\[
\bigl( \sigma\in\cO_n(0), \: \psi\in\cO_m(0), \: \rho\in\cO_d(0), \:
\phi\in\cO_\kappa(0), \: \chi\in\cO_{s-n-m-d-\kappa}(0) \bigr)
\]
such that $\nu$ depends analytically on $(\sigma,\psi,\rho,\phi,\chi)$, the
point $\nu=\nu^0$ corresponds to $\sigma=0$, $\psi=0$, $\rho=0$, $\phi=0$,
$\chi=0$, and
\begin{align*}
H\bigl(0, \, \nu(\sigma,\psi,\rho,\phi,\chi)\bigr) &\equiv \omega+\sigma, &
P\bigl(0, \, \nu(\sigma,\psi,\rho,\phi,\chi)\bigr) &\equiv \psi, \\
\beta\bigl(\nu(\sigma,\psi,\rho,\phi,\chi)\bigr) &\equiv \beta^0+\rho, &
\alpha_+\bigl(\nu(\sigma,\psi,\rho,\phi,\chi)\bigr) &\equiv \alpha^0_++\phi.
\end{align*}
Since the functions $H$ and $P$ are even in $y$, we have
\begin{equation}
H\bigl(y, \, \nu(\sigma,\psi,\rho,\phi,\chi)\bigr)-\omega-\sigma=O_2(y),
\qquad P\bigl(y, \, \nu(\sigma,\psi,\rho,\phi,\chi)\bigr)-\psi=O_2(y).
\label{eqeven}
\end{equation}

Following Moser's simple but very efficient trick \cite[Section~6b)]{M67}, set
for $\vare>0$
\[
y=\sqsq Y, \quad z=\sqsq Z, \quad \psi=\sqsq\Psi,
\]
where the variable $Y$ ranges in a certain fixed ($\vare$-independent) domain
in $\mR^m$ containing $0$ and symmetric with respect to $0$, the variable $Z$
ranges in a certain fixed ($\vare$-independent) domain in $\mR^{2p}$
containing $0$ and invariant under the linear involution $Z\mapsto KZ$, and
the parameter $\Psi$ ranges in a certain fixed ($\vare$-independent)
neighborhood of $0$ in $\mR^m$. It is not hard to verify using the
relation~\eqref{eqeven} for $P$ that in the variables $(x,Y,Z)$,
systems~\eqref{eqsystem} take the form
\begin{equation}
\begin{aligned}
\dot{x} &= \omega+\sigma+
\sqsq\Tf(x,Y,Z,\sigma,\sqsq\Psi,\rho,\phi,\chi,\sqsq), \\
\dot{Y} &= \Psi+
\sqsq\Tg(x,Y,Z,\sigma,\sqsq\Psi,\rho,\phi,\chi,\sqsq), \\
\dot{Z} &= Q\bigl(0, \, \nu(\sigma,0,\rho,\phi,\chi)\bigr)Z+
\sqsq\Th(x,Y,Z,\sigma,\Psi,\rho,\phi,\chi,\sqsq)
\end{aligned}
\label{eqtilde}
\end{equation}
with functions $\Tf$, $\Tg$, and $\Th$ analytic in all their arguments for
$\vare\geq 0$ (not merely for $\vare>0$) sufficiently small. One cannot write
$\Th(x,Y,Z,\sigma,\sqsq\Psi,\rho,\phi,\chi,\sqsq)$ because $\sqsq\Th$
incorporates the term
$Q\bigl(0, \, \nu(\sigma,\sqsq\Psi,\rho,\phi,\chi)\bigr)Z -
Q\bigl(0, \, \nu(\sigma,0,\rho,\phi,\chi)\bigr)Z$. Systems~\eqref{eqtilde}
are reversible with respect to the involution $G:(x,Y,Z)\mapsto(-x,-Y,KZ)$.

Obviously,
\[
Q\bigl(0, \, \nu(\sigma,0,\rho,\phi,\chi)\bigr)=
\fT(d_1,d_2,d_3;\alpha^\vee,\beta^0+\rho)
\]
where $\alpha^\vee_+=\alpha^0_++\phi$ and
$\alpha^\vee_-=\alpha_-\bigl(\nu(\sigma,0,\rho,\phi,\chi)\bigr)$. Introduce a
new parameter $\Phi\in\cO_{d_1+d_3-\kappa}(0)$ where the neighborhood
$\cO_{d_1+d_3-\kappa}(0)$ does not depend on $\vare$. Let
\[
\Lambda(\Phi)=\fT(d_1,d_2,d_3;\alpha^\star,\beta^0)
\]
with $\alpha^\star_+=\alpha^0_+$ and $\alpha^\star_-=\alpha^0_-+\Phi$.

Taking Lemmas~\ref{lemcent} and~\ref{lemfR} into account, apply Moser's
Theorem~\ref{thmain} to the family of systems
\begin{equation}
\begin{aligned}
\dot{x} &= \omega+\sqsq\Tf(x,Y,Z,\sigma,\sqsq\Psi,\rho,\phi,\chi,\sqsq), \\
\dot{Y} &= \sqsq\Tg(x,Y,Z,\sigma,\sqsq\Psi,\rho,\phi,\chi,\sqsq), \\
\dot{Z} &= \Lambda(\Phi)Z+
\sqsq\Th(x,Y,Z,\sigma,\Psi,\rho,\phi,\chi,\sqsq)
\end{aligned}
\label{eqPhi}
\end{equation}
reversible with respect to the involution $G:(x,Y,Z)\mapsto(-x,-Y,KZ)$ and
depending on the parameters
$(\sigma,\Psi,\rho,\phi,\chi,\Phi)\in\cO_{s+d_1+d_3-\kappa}(0)$. For $\vare$
sufficiently small, we obtain modified systems
\begin{equation}
\begin{aligned}
\dot{x} &= \omega+\sqsq\Tf(x,Y,Z,\sigma,\sqsq\Psi,\rho,\phi,\chi,\sqsq) +
\sqsq\lambda(\sigma,\Psi,\rho,\phi,\chi,\Phi,\sqsq), \\
\dot{Y} &= \sqsq\Tg(x,Y,Z,\sigma,\sqsq\Psi,\rho,\phi,\chi,\sqsq) +
\sqsq\mu(\sigma,\Psi,\rho,\phi,\chi,\Phi,\sqsq), \\
\dot{Z} &= \Lambda(\Phi)Z+
\sqsq\Th(x,Y,Z,\sigma,\Psi,\rho,\phi,\chi,\sqsq) +
\sqsq M(\sigma,\Psi,\rho,\phi,\chi,\Phi,\sqsq)Z,
\end{aligned}
\label{eqget}
\end{equation}
where
\[
M=\fT(d_1,d_2,d_3;q,r)
\]
with
\[
q=q(\sigma,\Psi,\rho,\phi,\chi,\Phi,\sqsq), \quad
r=r(\sigma,\Psi,\rho,\phi,\chi,\Phi,\sqsq).
\]
Here the functions $\lambda$, $\mu$, $q$, $r$ are determined uniquely and
analytic in all their arguments, the values of these functions ranging in
$\mR^n$, $\mR^m$, $\mR^{d_1+d_3}$, $\mR^d$, respectively. After the coordinate
change
\begin{align*}
x &= \xi+\sqsq A(\xi,\sigma,\Psi,\rho,\phi,\chi,\Phi,\sqsq), \\
\begin{pmatrix} Y \\ Z \end{pmatrix} &=
\begin{pmatrix} \eta \\ \zeta \end{pmatrix} +
\sqsq B(\xi,\sigma,\Psi,\rho,\phi,\chi,\Phi,\sqsq) +
\sqsq\bC(\xi,\sigma,\Psi,\rho,\phi,\chi,\Phi,\sqsq)
\begin{pmatrix} \eta \\ \zeta \end{pmatrix}
\end{align*}
[\,$\xi\in\mT^n$, $\eta\in\cO_m(0)$, and $\zeta\in\cO_{2p}(0)$ being new phase
space variables], system~\eqref{eqget} takes the form
\begin{equation}
\begin{aligned}
\dot{\xi} &= \omega+\sqsq O(\eta,\zeta), \\
\dot{\eta} &= \sqsq O_2(\eta,\zeta), \\
\dot{\zeta} &= \Lambda(\Phi)\zeta+\sqsq O_2(\eta,\zeta)
\end{aligned}
\label{eqgoal}
\end{equation}
with the right-hand sides analytic in $\xi$, $\eta$, $\zeta$, $\sigma$,
$\Psi$, $\rho$, $\phi$, $\chi$, $\Phi$, and $\sqvare$. This coordinate change
commutes with the involution $G$. The functions $A$, $B$, and $\bC$ are
analytic in all their arguments, the values of these functions ranging in
$\mR^n$, $\mR^{m+2p}$, and $\gl(m+2p,\mR)$, respectively.

Compared with the notation of Section~\ref{Moser}, here
\begin{itemize}
\item $N=m+2p$,
\item $(Y,Z)$ plays the role of $X$,
\item $(\eta,\zeta)$ plays the role of $\Xi$,
\item $s+d_1+d_3-\kappa$ plays the role of $s$ and $\cO_{s+d_1+d_3-\kappa}(0)$
plays the role of $\cN$,
\item $(\sigma,\Psi,\rho,\phi,\chi,\Phi)$ plays the role of $\nu$,
\item $\sqvare$ plays the role of $\vare$,
\item $\Tf$ plays the role of $f$,
\item $\bigl(\,\Tg,\Th\,\bigr)$ plays the role of $F$,
\item the block diagonal matrix with blocks $0_{m\times m}$ and
$\Lambda(\Phi)$ plays the role of $\bOmega(\nu)$,
\item $\beta^0$ plays the role of $\beta$,
\item $\bigl(\, \mu,\underbrace{0,\ldots,0}_{2p} \,\bigr)$ plays the role of
$\mu$,
\item the block diagonal matrix with blocks $0_{m\times m}$ and $M$ plays the
role of $\bM$.
\end{itemize}

Now one can use $\nu$ and $\alpha_-$ to ``compensate'' for the modifying terms
$\lambda$, $\mu$, and $M$. Observe that the systems~\eqref{eqtilde}
and~\eqref{eqget} coincide if $\sqsq\lambda=\sigma$, $\sqsq\mu=\Psi$, and
$\Lambda(\Phi)+\sqsq M=Q\bigl(0, \, \nu(\sigma,0,\rho,\phi,\chi)\bigr)$, i.e.
\begin{align*}
\sqsq\lambda(\sigma,\Psi,\rho,\phi,\chi,\Phi,\sqsq) &= \sigma, \\
\sqsq\mu(\sigma,\Psi,\rho,\phi,\chi,\Phi,\sqsq) &= \Psi, \\
\sqsq r(\sigma,\Psi,\rho,\phi,\chi,\Phi,\sqsq) &= \rho, \\
\sqsq q_+(\sigma,\Psi,\rho,\phi,\chi,\Phi,\sqsq) &= \phi, \\
\alpha^0_-+\Phi+\sqsq q_-(\sigma,\Psi,\rho,\phi,\chi,\Phi,\sqsq) &=
\alpha_-\bigl(\nu(\sigma,0,\rho,\phi,\chi)\bigr).
\end{align*}
According to the implicit function theorem, for sufficiently small $\vare$ and
any $\chi$, one can solve this system of equations with respect to $\sigma$,
$\Psi$, $\rho$, $\phi$, and $\Phi$:
\begin{align*}
\sigma &= \sqsq\Sigma^1(\chi,\sqsq), \\
\Psi &= \sqsq\Sigma^2(\chi,\sqsq), \\
\rho &= \sqsq\Sigma^3(\chi,\sqsq), \\
\phi &= \sqsq\Sigma^4(\chi,\sqsq), \\
\Phi &= \alpha_-\bigl(\nu(0,0,0,0,\chi)\bigr)-\alpha^0_-+\sqsq\Pi(\chi,\sqsq)
\end{align*}
with analytic functions $\Sigma^1$, $\Sigma^2$, $\Sigma^3$, $\Sigma^4$, and
$\Pi$.

We conclude that for $\vare>0$ sufficiently small and any
$\chi\in\cO_{s-n-m-d-\kappa}(0)$ with
\begin{equation}
\begin{aligned}
\sigma &= \sqsq\Sigma^1(\chi,\sqsq), & \psi &= \vare\Sigma^2(\chi,\sqsq), \\
\rho &= \sqsq\Sigma^3(\chi,\sqsq), & \phi &= \sqsq\Sigma^4(\chi,\sqsq),
\end{aligned}
\label{eqsurf}
\end{equation}
the coordinate transformation
\begin{equation}
\begin{aligned}
x &= \xi+\sqsq A(\xi,\sigma,\Psi,\rho,\phi,\chi,\Phi,\sqsq), \\
\begin{pmatrix} y \\ z \end{pmatrix} &=
\sqsq\begin{pmatrix} \eta \\ \zeta \end{pmatrix} +
\vare B(\xi,\sigma,\Psi,\rho,\phi,\chi,\Phi,\sqsq) +
\vare\bC(\xi,\sigma,\Psi,\rho,\phi,\chi,\Phi,\sqsq)
\begin{pmatrix} \eta \\ \zeta \end{pmatrix}
\end{aligned}
\label{eqABbC}
\end{equation}
with
\[
\Psi=\psi/\sqvare, \quad
\Phi=\alpha_-\bigl(\nu(0,0,0,0,\chi)\bigr)-\alpha^0_-+\sqsq\Pi(\chi,\sqsq)
\]
casts the original system~\eqref{eqsystem} into a system of the
form~\eqref{eqgoal} with the right-hand sides analytic in $\xi$, $\eta$,
$\zeta$, $\sigma$, $\Psi$, $\rho$, $\phi$, $\chi$, $\Phi$, and $\sqvare$.
Clearly, the $n$-torus $\{\eta=0, \, \zeta=0\}$ is invariant under the flow of
system~\eqref{eqgoal}, and the dynamics on this torus determined by the
equation $\dot{\xi}=\omega$ is quasi-periodic with frequency vector $\omega$.
This torus is also reducible with the Floquet matrix
\[
\begin{pmatrix}
0_{m\times m} & 0_{m\times 2p} \\ 0_{2p\times m} & \Lambda(\Phi)
\end{pmatrix}
\]
whose eigenvalues have the form~\eqref{eqprime} with
\[
\alpha'_+=\alpha^0_+, \qquad
\alpha'_-=\alpha^0_-+\Phi=
\alpha_-\bigl(\nu(0,0,0,0,\chi)\bigr)+\sqsq\Pi(\chi,\sqsq).
\]
Since the coordinate change~\eqref{eqABbC} commutes with the involution $G$,
the torus $\{\eta=0, \, \zeta=0\}$ is invariant under this involution as well,
and the restriction of $G$ to this torus has the form $\xi\mapsto-\xi$.

In the original coordinates $(x,y,z)$, the torus $\{\eta=0, \, \zeta=0\}$ is
given by the equations
\begin{align*}
x &= \xi+\sqsq A(\xi,\sigma,\Psi,\rho,\phi,\chi,\Phi,\sqsq), \\
\begin{pmatrix} y \\ z \end{pmatrix} &=
\vare B(\xi,\sigma,\Psi,\rho,\phi,\chi,\Phi,\sqsq)
\end{align*}
($\xi\in\mT^n$) and depends analytically on $\chi$ and $\sqvare$ for
$\vare\geq 0$ (not merely for $\vare>0$) sufficiently small.
Equations~\eqref{eqsurf} determine the desired surface $\cS_\vare$. The
surface $\cS_0$ has the form $\{\sigma=0, \, \psi=0, \, \rho=0, \, \phi=0\}$
and passes through the point $\nu^0$ corresponding to $\chi=0$. For any
$\chi$, the torus $\{\eta=0, \, \zeta=0\}$ for $\vare=0$ coincides with
$\{y=0, \, z=0\}$.
\end{proof}

Most probably, the invariant $n$-tori in the setting of Theorem~\ref{thC2} in
fact depend analytically on $\vare$, not merely on $\sqvare$, but it would be
hardly possible to prove that within Moser's approach.

Theorem~\ref{thC2} describes the persistence of the unperturbed reducible
invariant torus $\{y=0, \, z=0, \, \nu=\nu^0\}$ with the preservation of
\begin{itemize}
\item the frequencies $\omega_1,\ldots,\omega_n$,
\item all the imaginary parts $\pm\beta^0_1,\ldots,\pm\beta^0_d$ of the
Floquet exponents,
\item an arbitrary subcollection (of length $\kappa$) of the pairs of the real
parts $\pm\alpha^0_1,\ldots,\pm\alpha^0_{d_1+d_3}$ of the Floquet exponents.
\end{itemize}
The situation resembles the so-called \emph{partial preservation of Floquet
exponents}~\cite{S07Stek} in the ``well developed'' contexts of KAM theory
(namely, the reversible context~1, the Hamiltonian context, the volume
preserving context, and the dissipative context). The preservation of
$\omega$, $\beta^0$, and $\alpha^0_+$ requires at least $n+m+d+\kappa$
external parameters. If $\kappa=0$ (one is not going to have any control over
the real parts of the Floquet exponents), then the minimal number of external
parameters is equal to $n+m+d=\fN-\dim\Fix G-d_1-d_3$ where $\fN=n+m+2p$ is
the phase space dimension. Indeed, since $\dim\Fix G=p$,
\[
\fN-\dim\Fix G-d_1-d_3=n+m+p-d_1-d_3=n+m+d_2+d_3=n+m+d.
\]
If $\kappa=d_1+d_3$ (all the Floquet exponents are to be completely
preserved), then the minimal number of external parameters is equal to
$n+m+d+d_1+d_3=\fN-\dim\Fix G$. In the latter case, it suffices to use Moser's
theorem with the matrix $\bOmega$ independent of $\nu$. Indeed, the parameter
$\Phi$ in~\eqref{eqPhi} is absent for $\kappa=d_1+d_3$. In
Section~\ref{Moser}, we allowed $\bOmega$ to depend on $\nu$ in order to relax
the preservation requirement for the real parts of the Floquet exponents
(cf.\ \cite[Section~6d)]{M67}) and accordingly to reduce the necessary number
of external parameters.

\section{The Reversible Context~1}\label{cont1}

We will also formulate the counterpart of Theorem~\ref{thC2} for the
reversible context~1. In the notation of Subsection~\ref{nota2}, let
$\delta=1$, so that the reversing involution is $G:(x,y,z)\mapsto(-x,y,Kz)$.
Consider a family of $G$-reversible systems on $\mT^n\times\cY\times\cZ$ of
the form
\begin{equation}
\begin{aligned}
\dot{x} &= H(y,\nu)+f^\sharp(x,y,z,\nu)+\vare f(x,y,z,\nu,\vare), \\
\dot{y} &= g^\sharp(x,y,z,\nu)+\vare g(x,y,z,\nu,\vare), \\
\dot{z} &= Q(y,\nu)z+h^\sharp(x,y,z,\nu)+\vare h(x,y,z,\nu,\vare)
\end{aligned}
\label{eqfamiliar}
\end{equation}
(with $2p\times 2p$ matrix-valued function $Q$), where $f^\sharp=O(z)$,
$g^\sharp=O_2(z)$, and $h^\sharp=O_2(z)$. Reversibility of~\eqref{eqfamiliar}
with respect to $G$ means that
\[
Q(y,\nu)K\equiv-KQ(y,\nu)
\]
and
\begin{align*}
f^\sharp(-x,y,Kz,\nu) &\equiv f^\sharp(x,y,z,\nu), &
f(-x,y,Kz,\nu,\vare) &\equiv f(x,y,z,\nu,\vare), \\
g^\sharp(-x,y,Kz,\nu) &\equiv -g^\sharp(x,y,z,\nu), &
g(-x,y,Kz,\nu,\vare) &\equiv -g(x,y,z,\nu,\vare), \\
h^\sharp(-x,y,Kz,\nu) &\equiv -Kh^\sharp(x,y,z,\nu), &
h(-x,y,Kz,\nu,\vare) &\equiv -Kh(x,y,z,\nu,\vare).
\end{align*}
All the functions $H$, $Q$, $f^\sharp$, $g^\sharp$, $h^\sharp$, $f$, $g$, and
$h$ are assumed to be analytic in all their arguments.

Let the spectrum of the matrix $Q(y,\nu)$ anti-commuting with $K$ be simple
for any $y\in\cY$ and $\nu\in\cN$, and
\[
Q(y,\nu)=\fT\bigl(d_1,d_2,d_3;\alpha(y,\nu),\beta(y,\nu)\bigr),
\]
where the numbers $d_1\geq 0$, $d_2\geq 0$, $d_3\geq 0$ do not depend on $y$
and $\nu$ ($d_1+d_2+2d_3=p$), $\alpha_k(y,\nu)>0$ for all
$1\leq k\leq d_1+d_3$, $y\in\cY$, $\nu\in\cN$, and $\beta_l(y,\nu)>0$ for all
$1\leq l\leq d_2+d_3$, $y\in\cY$, $\nu\in\cN$. Introduce the notation
$d_2+d_3=d$.

Fix an \emph{arbitrary} (possibly, empty) subset of indices
$\fZ\subset\{1;2;\ldots;d_1+d_3\}$ consisting of $\kappa$ elements
($0\leq\kappa\leq d_1+d_3$). We will write $\alpha_+=(\alpha_k \mid k\in\fZ)$
and $\alpha_-=(\alpha_k \mid k\notin\fZ)$ and also use similar notation for
vector quantities in $\mR^{d_1+d_3}$ denoted by $\alpha^0$ and $\alpha'$.

\begin{thm}\label{thC1}
Suppose that
\begin{itemize}
\item $m+s\geq n+d+\kappa$,
\item for some $y^0\in\cY$ and $\nu^0\in\cN$, the vectors
$\omega=H(y^0,\nu^0)\in\mR^n$ and $\beta^0=\beta(y^0,\nu^0)\in\mR^d$ satisfy
the following Diophantine condition: there exist constants $\tau>n-1$ and
$\gamma>0$ such that the inequalities~\eqref{eqnewDioph} hold for all
$j\in\mZ^n\setminus\{0\}$ and $J\in\mZ^d$, $|J|\leq 2$,
\item the mapping
\[
(y,\nu)\mapsto\bigl( H(y,\nu), \, \beta(y,\nu), \, \alpha_+(y,\nu) \bigr)
\]
is submersive at point $(y^0,\nu^0)$, i.e.
\[
\left. \rank\frac{
\partial\bigl( H(y,\nu), \, \beta(y,\nu), \, \alpha_+(y,\nu) \bigr)
}{\partial(y,\nu)} \right|_{y=y^0, \, \nu=\nu^0} = n+d+\kappa.
\]
\end{itemize}
Then for sufficiently small $\vare$, in $\mT^n\times\cY\times\cZ\times\cN$,
there exists an $(m+s-n-d-\kappa)$-parameter analytic family of analytic
reducible invariant $n$-tori of systems~\eqref{eqfamiliar}. These tori carry
quasi-periodic motions with frequency vector $\omega$. The Floquet exponents
of these tori have the form
\begin{gather*}
\underbrace{0,\ldots,0}_m\,, \qquad
\pm\alpha'_1,\ldots,\pm\alpha'_{d_1}, \qquad
\pm i\beta^0_1,\ldots,\pm i\beta^0_{d_2}, \\
\pm\alpha'_{d_1+1}\pm i\beta^0_{d_2+1},\ldots,
\pm\alpha'_{d_1+d_3}\pm i\beta^0_{d_2+d_3}.
\end{gather*}
For each torus, $\alpha'_k>0$ for all $1\leq k\leq d_1+d_3$ and
$\alpha'_+=\alpha^0_+$, i.e., $\alpha'_k=\alpha_k(y^0,\nu^0)$ for $k\in\fZ$
\textup{[}here $\alpha^0=\alpha(y^0,\nu^0)\in\mR^{d_1+d_3}$\textup{]}. These
tori and the numbers $\alpha'_k$, $k\notin\fZ$ \textup{(}which vary along the
family of the tori\textup{)}, depend analytically on $\sqvare$. At $\vare=0$,
all the tori have the form $\{y=\const, \, z=0, \, \nu=\const\}$ and among
them, there is the torus $\{y=y^0, \, z=0, \, \nu=\nu^0\}$.
\end{thm}

The proof of this theorem is analogous to that of Theorem~\ref{thC2}. The main
idea of the proof is to use $y$, $\nu$, and $\alpha_-$ to ``compensate'' for
the suitable modifying terms. Of course, by now there are known much deeper
results concerning invariant $n$-tori of systems similar to~\eqref{eqfamiliar}
\cite{B91,BCHV09,BHN07,BH95,BHS96Gro,BHS96LNM,L01,Sch87,S95Cha,S98,S06,%
S07DCDS,S07Stek,WX09,WXZ10,WZX11,W01,X04,Z08}. We have given the formulation
of Theorem~\ref{thC1} just to emphasize parallelism in applications of
Theorem~\ref{thmain} to the non-extreme reversible contexts~1 and~2.


\begin{thebibliography}{99}

\bibitem{A84}
V.~I.~Arnold, \emph{Reversible systems}, Nonlinear and turbulent processes in
physics, vol.~3 (Kiev, 1983), Harwood Academic Publ., Chur, 1984,
pp.~1161--1174. MR~0824779

\bibitem{AKN06}
V.~I.~Arnold, V.~V.~Kozlov, and A.~I.~Neishtadt, \emph{Mathematical aspects of
classical and celestial mechanics}, 3rd ed., Encyclop{\ae}dia of Mathematical
Sciences, vol.~3, Springer-Verlag, Berlin, 2006. MR~2269239

\bibitem{AS86}
V.~I.~Arnold and M.~B.~Sevryuk, \emph{Oscillations and bifurcations in
reversible systems}, Nonlinear phenomena in plasma physics and hydrodynamics,
Mir Publishers, Moscow, 1986, pp.~31--64.

\bibitem{B73}
Yu.~N.~Bibikov, \emph{A sharpening of a theorem of Moser}, Dokl. Akad. Nauk
SSSR \textbf{213} (1973), no.~4, 766--769 (Russian). MR~0333359. English
translation: Soviet Math. Dokl. \textbf{14} (1973), no.~6, 1769--1773.

\bibitem{B91}
Yu.~N.~Bibikov, \emph{Multifrequency nonlinear oscillations and their
bifurcations}, Leningrad Univ. Press, Leningrad, 1991 (Russian). MR~1126680

\bibitem{Br72}
G.~E.~Bredon, \emph{Introduction to compact transformation groups}, Pure and
Applied Mathematics, vol.~46, Academic Press, New York, 1972. MR~0413144

\bibitem{BCHV09}
H.~W.~Broer, M.~C.~Ciocci, H.~Han{\ss}mann, and A.~Vanderbauwhede,
\emph{Quasi-periodic stability of normally resonant tori}, Phys.~D
\textbf{238} (2009), no.~3, 309--318. MR~2590451

\bibitem{BHN07}
H.~W.~Broer, J.~Hoo, and V.~Naudot, \emph{Normal linear stability of
quasi-periodic tori}, J.~Differential Equations \textbf{232} (2007), no.~2,
355--418. MR~2286385

\bibitem{BH95}
H.~W.~Broer and G.~B.~Huitema, \emph{Unfoldings of quasi-periodic tori in
reversible systems}, J.~Dynam. Differential Equations \textbf{7} (1995),
no.~1, 191--212. MR~1321710

\bibitem{BHS96Gro}
H.~W.~Broer, G.~B.~Huitema, and M.~B.~Sevryuk, \emph{Families of
quasi-periodic motions in dynamical systems depending on parameters},
Nonlinear dynamical systems and chaos (Groningen, 1995), Progr. Nonlinear
Differential Equations Appl., vol.~19, Birkh\"auser, Basel, 1996,
pp.~171--211. MR~1391497

\bibitem{BHS96LNM}
H.~W.~Broer, G.~B.~Huitema, and M.~B.~Sevryuk, \emph{Quasi-periodic motions in
families of dynamical systems. Order amidst chaos}, Lecture Notes in
Mathematics, vol.~1645, Springer-Verlag, Berlin, 1996. MR~1484969

\bibitem{BHTB90}
H.~W.~Broer, G.~B.~Huitema, F.~Takens, and B.~L.~J.~Braaksma, \emph{Unfoldings
and bifurcations of quasi-periodic tori}, Mem. Amer. Math. Soc. \textbf{83}
(1990), no.~421, $\text{viii}+175$~pp. MR~1041003

\bibitem{BS10}
H.~W.~Broer and M.~B.~Sevryuk, \emph{KAM theory:\ quasi-periodicity in
dynamical systems}, Handbook of Dynamical Systems, vol.~3, Elsevier~B.V.,
Amsterdam, 2010, chapter~6, pp.~249--344.

\bibitem{CF64}
P.~E.~Conner and E.~E.~Floyd, \emph{Differentiable periodic
maps}, Ergebnisse der Mathematik und ihrer Grenzgebiete, N.~F., Band~33,
Academic Press Inc., New York; Springer-Verlag, Berlin, 1964. MR~0176478

\bibitem{CGP11}
L.~Corsi, G.~Gentile, and M.~Procesi, \emph{KAM theory in configuration space
and cancellations in the Lindstedt series}, Comm. Math. Phys. \textbf{302}
(2011), no.~2, 359--402. MR~2770017

\bibitem{H96}
I.~Hoveijn, \emph{Versal deformations and normal forms for reversible and
Hamiltonian linear systems}, J.~Differential Equations \textbf{126} (1996),
no.~2, 408--442. MR~1383984

\bibitem{LR98}
J.~S.~W.~Lamb and J.~A.~G.~Roberts, \emph{Time-reversal symmetry in dynamical
systems:\ a survey}, Phys.~D \textbf{112} (1998), no.~1--2, 1--39. MR~1605826

\bibitem{L01}
B.~Liu, \emph{On lower dimensional invariant tori in reversible systems},
J.~Differential Equations \textbf{176} (2001), no.~1, 158--194. MR~1861186

\bibitem{MZ74}
D.~Montgomery and L.~Zippin, \emph{Topological transformation groups},
Robert E. Krieger Publishing Co., Huntington, NY, 1974. MR~0379739. Reprint
of the 1955 original.

\bibitem{M65}
J.~Moser, \emph{Combination tones for Duffing's equation}, Comm. Pure Appl.
Math. \textbf{18} (1965), no.~1--2, 167--181. MR~0179430

\bibitem{M66}
J.~Moser, \emph{On the theory of quasiperiodic motions}, SIAM Rev. \textbf{8}
(1966), no.~2, 145--172. MR~0203160

\bibitem{M67}
J.~Moser, \emph{Convergent series expansions for quasi-periodic motions},
Math. Ann. \textbf{169} (1967), no.~1, 136--176. MR~0208078

\bibitem{M73}
J.~Moser, \emph{Stable and random motions in dynamical systems. With special
emphasis on celestial mechanics}, Princeton Landmarks in Mathematics,
Princeton Univ. Press, Princeton, NJ, 2001. MR~1829194. Reprint of the 1973
original.

\bibitem{QS93}
G.~R.~W.~Quispel and M.~B.~Sevryuk, \emph{KAM theorems for the product of two
involutions of different types}, Chaos \textbf{3} (1993), no.~4, 757--769.
MR~1256318

\bibitem{RQ92}
J.~A.~G.~Roberts and G.~R.~W.~Quispel, \emph{Chaos and time-reversal symmetry.
Order and chaos in reversible dynamical systems}, Phys. Rep. \textbf{216}
(1992), no.~2--3, 63--177. MR~1173588

\bibitem{Sch87}
J.~Scheurle, \emph{Bifurcation of quasi-periodic solutions from equilibrium
points of reversible dynamical systems}, Arch. Rational Mech. Anal.
\textbf{97} (1987), no.~2, 103--139. MR~0860303

\bibitem{S86}
M.~B.~Sevryuk, \emph{Reversible systems}, Lecture Notes in Mathematics,
vol.~1211, Springer-Verlag, Berlin, 1986. MR~0871875

\bibitem{S91}
M.~B.~Sevryuk, \emph{Lower-dimensional tori in reversible systems}, Chaos
\textbf{1} (1991), no.~2, 160--167. MR~1135903

\bibitem{S92}
M.~B.~Sevryuk, \emph{Linear reversible systems and their versal deformations},
Trudy Sem. Petrovsk. no.~15 (1991), 33--54 (Russian). MR~1294389. English
translation: J.~Soviet Math. \textbf{60} (1992), no.~5, 1663--1680. MR~1181098

\bibitem{S95RMS}
M.~B.~Sevryuk, \emph{Some problems in KAM theory:\ conditionally periodic
motions in typical systems}, Uspekhi Mat. Nauk \textbf{50} (1995), no.~2,
111--124 (Russian). MR~1339266. English translation: Russian Math. Surveys
\textbf{50} (1995), no.~2, 341--353.

\bibitem{S95Cha}
M.~B.~Sevryuk, \emph{The iteration-approximation decoupling in the reversible
KAM theory}, Chaos \textbf{5} (1995), no.~3, 552--565. MR~1350654

\bibitem{S98}
M.~B.~Sevryuk, \emph{The finite-dimensional reversible KAM theory}, Phys.~D
\textbf{112} (1998), no.~1--2, 132--147. MR~1605834

\bibitem{S06}
M.~B.~Sevryuk, \emph{Partial preservation of frequencies in KAM theory},
Nonlinearity \textbf{19} (2006), no.~5, 1099--1140. MR~2221801

\bibitem{S07DCDS}
M.~B.~Sevryuk, \emph{Invariant tori in quasi-periodic non-autonomous dynamical
systems via Herman's method}, Discrete Contin. Dyn. Syst. \textbf{18} (2007),
no.~2--3, 569--595. MR~2291912

\bibitem{S07Stek}
M.~B.~Sevryuk, \emph{Partial preservation of frequencies and Floquet exponents
in KAM theory}, Trudy Mat. Inst. Steklova \textbf{259} (2007), 174--202
(Russian). MR~2433684. English translation: Proc. Steklov Inst. Math.
\textbf{259} (2007), 167--195.

\bibitem{S08}
M.~B.~Sevryuk, \emph{KAM tori:\ persistence and smoothness}, Nonlinearity
\textbf{21} (2008), no.~10, T177--T185. MR~2439472

\bibitem{S11}
M.~B.~Sevryuk, \emph{The reversible context~2 in KAM theory:\ the first
steps}, Regul. Chaotic Dyn. \textbf{16} (2011), no.~1--2, 24--38. MR~2774376

\bibitem{Sh93}
Ch.-W.~Shih, \emph{Normal forms and versal deformations of linear involutive
dynamical systems}, Chinese J. Math. \textbf{21} (1993), no.~4, 333--347.
MR~1247555

\bibitem{Wa10}
F.~Wagener, \emph{A parametrised version of Moser's modifying terms theorem},
Discrete Contin. Dyn. Syst. Ser.~S \textbf{3} (2010), no.~4, 719--768.
MR~2684071

\bibitem{WX09}
X.~Wang and J.~Xu, \emph{Gevrey-smoothness of invariant tori for analytic
reversible systems under R\"ussmann's non-degeneracy condition}, Discrete
Contin. Dyn. Syst. \textbf{25} (2009), no.~2, 701--718. MR~2525200

\bibitem{WXZ10}
X.~Wang, J.~Xu, and D.~Zhang, \emph{Persistence of lower dimensional elliptic
invariant tori for a class of nearly integrable reversible systems}, Discrete
Contin. Dyn. Syst. Ser.~B \textbf{14} (2010), no.~3, 1237--1249. MR~2670193

\bibitem{WZX11}
X.~Wang, D.~Zhang, and J.~Xu, \emph{Persistence of lower dimensional tori for
a class of nearly integrable reversible systems}, Acta Appl. Math.
\textbf{115} (2011), no.~2, 193--207.

\bibitem{W01}
B.~Wei, \emph{Perturbations of lower dimensional tori in the resonant zone for
reversible systems}, J.~Math. Anal. Appl. \textbf{253} (2001), no.~2,
558--577. MR~1808153

\bibitem{X04}
J.~Xu, \emph{Normal form of reversible systems and persistence of lower
dimensional tori under weaker nonresonance conditions}, SIAM J. Math. Anal.
\textbf{36} (2004), no.~1, 233--255. MR~2083860

\bibitem{Z08}
J.~Zhang, \emph{On lower dimensional invariant tori in $C^d$ reversible
systems}, Chin. Ann. Math. Ser.~B \textbf{29} (2008), no.~5, 459--486.
MR~2447481

\end{thebibliography}
\end{document}